\documentclass[runningheads,a4paper]{llncs}
\usepackage{amssymb}
\usepackage{graphicx}
\usepackage{hyperref}
\usepackage{mathabx}
\usepackage{url}
\usepackage{array}   
\newcommand{\keywords}[1]{\par\addvspace\baselineskip
\noindent\keywordname\enspace\ignorespaces#1}
\usepackage{amssymb,amstext}

\def\F{\mathcal F}
\def\P{\mathcal P}
\def\M{\mathcal M}
\def\I{\mathcal I}

\def\prev{\mathbb{P}}

\def\S{\mathcal{S}}

\def\C{\mathcal{C}}

\newcolumntype{L}{>{$}l<{$}}
\newcolumntype{C}{>{$}c<{$}}
\title{Conjunction of Conditional Events and T-norms}

\author{Angelo Gilio\inst{1}\thanks{Both authors contributed equally to the article and are listed alphabetically.}\,\thanks{Retired} \and  Giuseppe Sanfilippo\inst{2}
	 $^{\star}$}
\institute{Department SBAI, University of Rome ``La Sapienza'', Rome, Italy
	\\ \email{angelo.gilio@sbai.uniroma1.it}
	\and
	Department of Mathematics and Computer Science,
	University of Palermo, Italy
	\\ \email{giuseppe.sanfilippo@unipa.it}
}
\authorrunning{Angelo Gilio and Giuseppe Sanfilippo}
\begin{document}
\maketitle
\begin{abstract}
We study the relationship between a notion of conjunction among conditional events, introduced	in recent papers, and the notion of Frank t-norm. By examining different cases, in the setting of coherence,  we show each time that the conjunction coincides with a suitable Frank t-norm. In particular, the conjunction may coincide with the Product t-norm, the Minimum t-norm, and Lukasiewicz t-norm.
We show by a counterexample, that the prevision assessments obtained by  Lukasiewicz t-norm  may be not coherent. Then, we give some conditions of coherence when using Lukasiewicz t-norm. 
\keywords{Coherence,  Conditional Event, Conjunction, Frank t-norm.}
\end{abstract}
\section{Introduction}
In this paper we use the coherence-based approach to probability of de Finetti (\cite{biazzo00,biazzo05,coletti02,CoSV13,CoSV15,defi36,definetti74,gilio02,gilio12ijar,gilio16,gilio13,PfSa19}). We use a notion of conjunction which, differently from other authors, is defined  as a  suitable conditional random quantity with values in the unit interval (see, e.g. \cite{GiSa13c,GiSa13a,GiSa14,GiSa19,SPOG18}).  
We study  the relationship between our notion of conjunction  and the notion of Frank t-norm. For some aspects which relate probability and Frank t-norm see, e.g., \cite{Tasso12,coletti14FSS,coletti04,Dubois86,Flaminio18,Navara05}. We show that, under the hypothesis of logical independence, if the prevision assessments involved with the conjunction $(A|H) \wedge (B|K)$ 
of two conditional events are coherent, then the prevision of the conjunction coincides, for a suitable $\lambda \in [0,+\infty]$, with the Frank t-norm $T_\lambda(x,y)$, where $x=P(A|H), y=P(B|K)$. Moreover, $(A|H) \wedge (B|K)= T_\lambda(A|H,B|K)$. Then, we consider the case $A=B$, by determining the set of all coherent assessment $(x,y,z)$ on $\{A|H,A|K, (A|H) \wedge (A|K)\}$. We show that, under coherence, it holds that  $(A|H) \wedge (A|K)= T_\lambda(A|H,A|K)$, where $\lambda \in [0,1]$.  We also study the particular case where $A=B$ and $HK=\emptyset$. Then, we consider conjunctions of three conditional events and we show that to make prevision assignments by means of the Product t-norm, or the Minimum t-norm, is coherent. Finally, we examine the Lukasiewicz t-norm and we show by a counterexample that coherence is in general not assured. We give some conditions for coherence when the prevision assessments are made by using the  Lukasiewicz t-norm.  

\section{Preliminary Notions and Results}
In our approach, given two events $A$ and $H$, with $H \neq \emptyset$, the conditional event $A|H$ is looked at as a three-valued logical entity which is true, or false, or void, according to whether $AH$ is true, or $\widebar{A}H$ is true, or $\widebar{H}$ is true. We observe that the conditional probability and/or conditional prevision values are assessed in the setting of coherence-based probabilistic approach.
 In numerical terms $A|H$ assumes one of the values $1$, or $0$, or $x$, where $x=P(A|H)$ represents the assessed degree of belief on $A|H$.  Then, $A|H=AH+x\widebar{H}$. 
Given a family $\F = \{X_1|H_1,\ldots,X_n|H_n\}$, for each $i \in \{1,\ldots,n\}$ we denote by $\{x_{i1}, \ldots,x_{ir_i}\}$ the set of possible values  of $X_i$ when  $H_i$ is true; then, for each $i$ and $j = 1, \ldots, r_i$, we set $A_{ij} = (X_i = x_{ij})$. 
We set $C_0 = \widebar{H}_1 \cdots \widebar{H}_n$ (it may be $C_0 = \emptyset$);
moreover, we denote by $C_1, \ldots, C_m$ the constituents
contained in $H_1\vee \cdots \vee H_n$. Hence
$\bigwedge_{i=1}^n(A_{i1} \vee \cdots \vee A_{ir_i} \vee
\widebar{H}_i) = \bigvee_{h = 0}^m C_h$.
With each $C_h,\, h \in \{1,\ldots,m\}$, we associate a vector
$Q_h=(q_{h1},\ldots,q_{hn})$, where $q_{hi}=x_{ij}$ if $C_h \subseteq
A_{ij},\, j=1,\ldots,r_i$, while $q_{hi}=\mu_i$ if $C_h \subseteq \widebar{H}_i$;
with $C_0$ it is associated  $Q_0=\M = (\mu_1,\ldots,\mu_n)$.
Denoting by $\I$ the convex hull of $Q_1, \ldots, Q_m$, the condition  $\M\in \I$ amounts to the existence of a vector $(\lambda_1,\ldots,\lambda_m)$ such that:
$ \sum_{h=1}^m \lambda_h Q_h = \M \,,\; \sum_{h=1}^m \lambda_h
= 1 \,,\; \lambda_h \geq 0 \,,\; \forall \, h$; in other words, $\M\in \I$ is equivalent to the solvability of the system $(\Sigma)$, associated with  $(\F,\M)$,
\begin{equation}\label{SYST-SIGMA}
(\Sigma) \quad
\begin{array}{ll}
\sum_{h=1}^m \lambda_h q_{hi} =
\mu_i \,,\; i \in\{1,\ldots,n\} \,, 
\sum_{h=1}^m \lambda_h = 1,\;\;\lambda_h \geq 0 \,,\;  \,h \in\{1,\ldots,m\}\,.
\end{array}
\end{equation}
Given the assessment $\M =(\mu_1,\ldots,\mu_n)$ on  $\F =
\{X_1|H_1,\ldots,X_n|H_n\}$, let $S$ be the set of solutions $\Lambda = (\lambda_1, \ldots,\lambda_m)$ of system $(\Sigma)$.   
We point out that the solvability of  system $(\Sigma)$  is a necessary (but not sufficient) condition for coherence of $\M$ on $\F$. When $(\Sigma)$ is solvable, that is  $S \neq \emptyset$, we define:
\begin{equation}\label{EQ:I0}
\begin{array}{ll}
I_0 = \{i : \max_{\Lambda \in S}  \sum_{h:C_h\subseteq H_i}\lambda_h= 0\},\;
\F_0 = \{X_i|H_i \,, i \in I_0\},\;\;  \M_0 = (\mu_i ,\, i \in I_0)\,.
\end{array}
\end{equation}
For what concerns the probabilistic meaning of $I_0$, it holds that  $i\in I_0$ if and only if the (unique) coherent extension of $\M$ to $H_i|(\bigvee_{j=1}^nH_j)$ is zero.
Then, the following theorem can be proved  (\cite[Theorem 3]{BiGS08})
\begin{theorem}\label{CNES-PREV-I_0-INT}{\rm [{\em Operative characterization of coherence}]
		A conditional prevision assessment ${\M} = (\mu_1,\ldots,\mu_n)$ on
		the family $\F = \{X_1|H_1,\ldots,X_n|H_n\}$ is coherent if
		and only if the following conditions are satisfied: \\
		(i) the system $(\Sigma)$ defined in (\ref{SYST-SIGMA}) is solvable; (ii) if $I_0 \neq \emptyset$, then $\M_0$ is coherent. }
\end{theorem}
Coherence can be related to  proper scoring rules (\cite{BiGS12,GiSa11a,LSA12,LSA15,LaSA18}).
\begin{definition}\label{CONJUNCTION}Given any pair of conditional events $A|H$ and $B|K$, with $P(A|H)=x$ and $P(B|K)=y$, their conjunction
		is the conditional random quantity $(A|H)\wedge(B|K)$, with $\prev[(A|H)\wedge(B|K)]=z$,  defined as
\begin{equation}\label{EQ:CONJUNCTION}
(A|H)\wedge(B|K) =\left\{\begin{array}{ll}
1, &\mbox{if $AHBK$ is true,}\\
0, &\mbox{if $\widebar{A}H\vee \widebar{B}K$ is true,}\\
x, &\mbox{if $\widebar{H}BK$ is true,}\\
y, &\mbox{if $AH\widebar{K}$ is true,}\\
z, &\mbox{if $\widebar{H}\,\widebar{K}$ is true}.
\end{array}
\right.
\end{equation}
\end{definition}
In betting terms, the prevision $z$ represents the amount you agree to pay, with the proviso that you will receive the quantity $(A|H)\wedge(B|K)$.
Different approaches to compounded conditionals, not based on coherence,  have been developed by other authors  (see, e.g., \cite{Kauf09,mcgee89}).
We  recall a result  which shows that Fr\'echet-Hoeffding bounds still hold for the conjunction of conditional events (\cite[Theorem~7]{GiSa14}).
\begin{theorem}\label{THM:FRECHET}{\rm
		Given any coherent assessment $(x,y)$ on $\{A|H, B|K\}$, with $A,H,B$, $K$ logically independent, $H\neq \emptyset, K\neq  \emptyset$, the extension $z = \mathbb{P}[(A|H) \wedge (B|K)]$ is coherent if and only if the following  Fr\'echet-Hoeffding bounds are satisfied:
		\begin{equation}\label{LOW-UPPER}
		\max\{x+y-1,0\} = z' \; \leq \; z \; \leq \; z'' = \min\{x,y\} \,.
		\end{equation}
}\end{theorem}
\begin{remark}
From Theorem \ref{THM:FRECHET}, as the assessment $(x,y)$ on $\{A|H,B|K\}$ is coherent for every $(x,y)\in[0,1]^2$, the set $\Pi$ of coherent assessments $(x,y,z)$ on $\{A|H,B|K,(A|H)\wedge(B|K)\}$ is 
\begin{equation}\label{EQ:PI2}
\begin{small}
\Pi=\{(x,y,z): (x,y)\in[0,1]^2, \max\{x+y-1,0\}  \leq z\leq \min\{x,y\}
\end{small}
\}.
\end{equation}
The set $\Pi$ is the tetrahedron with vertices the points $(1,1,1)$, $(1,0,0)$, $(0,1,0)$, $(0,0,0)$. For other definition of conjunctions, where the conjunction is a conditional event, some results on lower and upper bounds have been given in \cite{SUM2018S}.
\end{remark}
\begin{definition}\label{DEF:CONGn}	Let be given  $n$ conditional events $E_1|H_1,\ldots,E_n|H_n$.
For each   subset $S$, with $\emptyset\neq S \subseteq \{1,\ldots,n\}$,  let $x_{S}$ be a  prevision assessment on $\bigwedge_{i\in S} (E_i|H_i)$.
The conjunction $\C_{1\cdots n}=(E_1|H_1) \wedge \cdots \wedge (E_n|H_n)$ is defined as
\begin{equation}\label{EQ:CF}
\begin{array}{lll}
\C_{1\cdots n}=
\left\{
\begin{array}{llll}
1, &\mbox{ if } \bigwedge_{i=1}^n E_iH_i, \mbox{ is true} \\
0, &\mbox{ if } \bigvee_{i=1}^n \widebar{E}_iH_i, \mbox{ is true}, \\
x_{S}, &\mbox{ if } \bigwedge_{i\in S} \widebar{H}_i\bigwedge_{i\notin S} E_i{H}_i\, \mbox{ is true}, \; \emptyset \neq S\subseteq \{1,2\ldots,n\}.
\end{array}
\right.
\end{array}
\end{equation}
\end{definition}
In particular, $\C_1=E_1|H_1$;
moreover, for $\S=\{i_1,\ldots,i_k\}\subseteq \{1,\ldots,n\}$, the conjunction $\bigwedge_{i\in S} (E_i|H_i)$ is denoted by $\C_{i_1\cdots i_k}$ and $x_{\S}$ is also denoted by $x_{i_1\cdots i_k}$.
Moreover, if $\S=\{i_1,\ldots,i_k\}\subseteq \{1,\ldots,n\}$, the conjunction $\bigwedge_{i\in S} (E_i|H_i)$ is denoted by $\C_{i_1\cdots i_k}$ and $x_{\S}$ is also denoted by $x_{i_1\cdots i_k}$. 
In the betting framework, you agree to pay $x_{1\cdots n}=\prev(	\C_{1\cdots n})$ with the proviso that you will receive: $1$, if all conditional events are true;  $0$, if at least one of the conditional events is false;  the prevision of the conjunction of that conditional events which are void,  otherwise. 
The operation of conjunction is associative and commutative. 
We observe that, based on  Definition \ref{DEF:CONGn}, when $n=3$ we obtain
\begin{equation}\label{EQ:CONJUNCTION3}
\begin{small}
\begin{array}{lll}
\C_{123}
=\left\{
\begin{array}{llll}
1, &\mbox{ if } E_1H_1E_2H_2E_3H_3 \mbox{ is true},\\
0, &\mbox{ if } \widebar{E}_1H_1 \vee \widebar{E}_2H_2 \vee \widebar{E}_3H_3 \mbox{ is true},\\
x_1,& \mbox{ if } \widebar{H}_1E_2H_2E_3H_3 \mbox{ is true},\\
x_2,& \mbox{ if } \widebar{H}_2E_1H_1E_3H_3 \mbox{ is true},\\
x_3, &\mbox{ if } \widebar{H}_3E_1H_1E_2H_2 \mbox{ is true}, \\
x_{12}, &\mbox{ if } \widebar{H}_1\widebar{H}_2E_3H_3 \mbox{ is true}, \\
x_{13}, &\mbox{ if } \widebar{H}_1\widebar{H}_3E_2H_2 \mbox{ is true}, \\
x_{23}, &\mbox{ if } \widebar{H}_2\widebar{H}_3E_1H_1 \mbox{ is true}, \\
x_{123}, &\mbox{ if } \widebar{H}_1\widebar{H}_2\widebar{H}_3 \mbox{ is true}. \\
\end{array}
\right.
\end{array}
\end{small}
\end{equation}
 We recall the following result (\cite[Theorem 15]{GiSa19}). 
\begin{theorem}\label{THM:PIFOR3}
Assume that  the events $E_1, E_2, E_3, H_1, H_2, H_3$  are logically independent, with $H_1\neq \emptyset, H_2\neq \emptyset, H_3\neq \emptyset$.
Then,	the set $\Pi$ of all coherent assessments  $\mathcal{M}=(x_1,x_2,x_3,x_{12},x_{13},x_{23},x_{123})$ on
$\F=\{\C_{1},\C_{2},\C_{3}, \C_{12}, \C_{13}, \C_{23}, \C_{123}\}$ is the set of points $(x_1,x_2,x_3,x_{12},x_{13},x_{23},x_{123})$ 
which satisfy the following  conditions
\begin{equation}
\small
\label{EQ:SYSTEMPISTATEMENT}
\left\{
\begin{array}{l}
(x_1,x_2,x_3)\in[0,1]^3,\\
\max\{x_1+x_2-1,x_{13}+x_{23}-x_3,0\}\leq x_{12}\leq \min\{x_1,x_2\},\\
\max\{x_1+x_3-1,x_{12}+x_{23}-x_2,0\}\leq x_{13}\leq \min\{x_1,x_3\},\\
\max\{x_2+x_3-1,x_{12}+x_{13}-x_1,0\}\leq x_{23}\leq \min\{x_2,x_3\},\\
1-x_1-x_2-x_3+x_{12}+x_{13}+x_{23}\geq 0,\\
x_{123}\geq \max\{0,x_{12}+x_{13}-x_1,x_{12}+x_{23}-x_2,x_{13}+x_{23}-x_3\},\\
x_{123}\leq  \min\{x_{12},x_{13},x_{23},1-x_1-x_2-x_3+x_{12}+x_{13}+x_{23}\}.
\end{array}
\right.
\end{equation}
\end{theorem}
\begin{remark}\label{REM:INEQPI}
As shown in  (\ref{EQ:SYSTEMPISTATEMENT}), the  coherence of  $(x_1,x_2,x_3,x_{12},x_{13},x_{23},x_{123})$ amounts to  the condition
\begin{equation}\label{EQ:INEQPI}
\begin{array}{ll}
\max\{0,x_{12}+x_{13}-x_1,x_{12}+x_{23}-x_2,x_{13}+x_{23}-x_3\}\,\;\leq  \;x_{123}\;\leq \\
\leq\;\; 
\min\{x_{12},x_{13},x_{23},1-x_1-x_2-x_3+x_{12}+x_{13}+x_{23}\}.
\end{array}
\end{equation}
Then, in particular, 
	the extension $x_{123}$ on $\C_{123}$ is coherent if and only if $x_{123}\in[x_{123}',x_{123}'']$, where
 $x_{123}'=\max\{0,x_{12}+x_{13}-x_1,x_{12}+x_{23}-x_2,x_{13}+x_{23}-x_3\}$, $
x_{123}''= \min\{x_{12},x_{13},x_{23},1-x_1-x_2-x_3+x_{12}+x_{13}+x_{23}\}.$
\end{remark}
Then, by   Theorem \ref{THM:PIFOR3} it follows \cite[Corollary 1]{GiSa19}
\begin{corollary}\label{COR:PIFOR3}
For any coherent assessment  $(x_1,x_2,x_3,x_{12},x_{13},x_{23})$ on
$\{\C_{1},\C_{2},\C_{3}, \C_{12}, \C_{13}, \C_{23}\}$
the extension $x_{123}$ on $\C_{123}$ is coherent if and only if $x_{123}\in[x_{123}',x_{123}'']$, where
\begin{equation}\label{EQ:INECOR}
\begin{array}{ll}
x_{123}'=\max\{0,x_{12}+x_{13}-x_1,x_{12}+x_{23}-x_2,x_{13}+x_{23}-x_3\},\\
x_{123}''= \min\{x_{12},x_{13},x_{23},1-x_1-x_2-x_3+x_{12}+x_{13}+x_{23}\}.
\end{array}
\end{equation}
\end{corollary}	
We recall that in case of logical dependencies, the set of all coherent assessments may be smaller than that one associated with  the  case of logical independence. 
However (see \cite[Theorem 16]{GiSa19}) the set of coherent assessments is the same when  $H_1=H_2=H_3=H$ (where possibly $H=\Omega$; see also \cite[p. 232]{Joe97}) and 
a corollary similar to Corollary \ref{COR:PIFOR3} also holds in this case. For a similar result  based on copulas see \cite{Dura08}.
\section{Representation by Frank t-norms for  $(A|H)\wedge(B|K)$}
We recall that for every $\lambda \in[0,+\infty]$ the Frank t-norm $T_{\lambda}:[0,1]^2\rightarrow [0,1]$ with parameter $\lambda$ is defined as
\begin{equation}\label{EQ:FRANK}
\begin{small}
T_{\lambda}(u,v)=\left\{\begin{array}{ll}
T_{M}(u,v)=\min\{u,v\}, & \text{ if } \lambda=0,\\
T_{P}(u,v)=uv, & \text{ if } \lambda=1,\\	  
T_{L}(u,v)=\max\{u+v-1,0\}, & \text{ if } \lambda=+\infty,\\	  	
\log_{\lambda}(1+\frac{(\lambda^u-1)(\lambda^v-1)}{\lambda-1}), & \text{ otherwise}.
\end{array}\right.
\end{small}
\end{equation}
We recall that  $T_{\lambda}$ is continuous   with respect to $\lambda$;  
moreover, for every  $\lambda \in[0,+\infty]$, it holds that $T_{L}(u,v)\leq T_{\lambda}(u,v) \leq T_{M}(u,v)$, for every $(u,v)\in[0,1]^2$ (see, e.g., \cite{KlMP00},\cite{KlMe05}). In the next result we study the relation between our notion of conjunction and t-norms.
\begin{theorem}\label{THM_TNORM2}
Let us consider the conjunction $(A|H)\wedge(B|K)$, with $A, B, H, K$ logically independent and with $P(A|H)=x$, $P(B|K)=y$.
Moreover, given any $\lambda  \in[0,+\infty]$, let $T_{\lambda}$ be the  Frank t-norm with parameter $\lambda$. Then,  the assessment
 $z=T_{\lambda}(x,y)$ on $(A|H)\wedge(B|K)$ is a coherent extension of $(x,y)$ on $\{A|H,B|K\}$; moreover   $(A|H)\wedge(B|K)=T_{\lambda}(A|H, B|K)$. Conversely, given any coherent extension  $z=\prev[(A|H)\wedge(B|K)]$ of $(x,y)$, there exists $\lambda \in[0,+\infty]$ such that $z=T_{\lambda}(x,y)$.
\end{theorem}
\begin{proof}
We observe that from Theorem \ref{THM:FRECHET},
for any  given  $\lambda$, the assessment $z=T_{\lambda}(x,y)$ is a coherent extension of $(x,y)$  on $\{A|H,B|K\}$. Moreover,
from (\ref{EQ:FRANK}) it holds that $T_{\lambda}(1,1)=1$,  $T_{\lambda}(u,0)=T_{\lambda}(0,v)=0$, $T_{\lambda}(u,1)=u$,   $T_{\lambda}(1,v)=v$. Hence,
\begin{equation}\label{EQ:FRANKBIS}
\begin{small}
		T_{\lambda}(A|H,B|K) =\left\{\begin{array}{ll}
	1, &\mbox{ if $AHBK$ is true,}\\
	0, &\mbox{ if  $\widebar{A}H$ is true or  $\widebar{B}K$ is true,}\\
	x, &\mbox{ if $\widebar{H}BK$ is true,}\\
	y, &\mbox{ if $\widebar{K}AH$ is true,}\\
	T_{\lambda}(x,y), &\mbox{ if $\widebar{H}\,\widebar{K}$ is true},
\end{array}
\right.
\end{small}	
\end{equation}
and, if we choose $z=T_{\lambda}(x,y)$, 
 from (\ref{EQ:CONJUNCTION}) and (\ref{EQ:FRANKBIS})   it follows that $(A|H)\wedge(B|K)=T_{\lambda}(A|H, B|K)$. \\ 
Conversely, given any coherent extension $z$ of $(x,y)$, there exists $\lambda$ such that $z=T_{\lambda}(x,y)$. Indeed, if $z=\min\{x,y\}$, then $\lambda=0$; if $z=\max\{x+y-1,0\}$, then $\lambda=+\infty$;
if $\max\{x+y-1,0\}<z<\min\{x,y\}
$, then by  continuity of $T_{\lambda}$   with respect to $\lambda$ 
it holds that
$z=T_{\lambda}(x,y)$ 
for some  $\lambda \in\,]0,\infty[$  (for instance, if $z=xy$, then $z=T_1(x,y)$) and hence $(A|H)\wedge(B|K)=T_{\lambda}(A|H, B|K)$.
\qed \end{proof}
\begin{remark}
As we can see from (\ref{EQ:CONJUNCTION}) and Theorem \ref{THM_TNORM2}, in case of logically independent events, if the assessed values $x,y,z$ are such that $z=T_{\lambda}(x,y)$ for a given $\lambda$, then the conjunction $(A|H)\wedge (B|K)=T_{\lambda}(A|H,B|K)$. For instance, if $z=T_1(x,y)=xy$, then 
$(A|H)\wedge (B|K)=T_{1}(A|H,B|K)=(A|H)\cdot (B|K)$. Conversely, if $(A|H)\wedge (B|K)=T_{\lambda}(A|H,B|K)$ for a given $\lambda$, then $z=T_{\lambda}(x,y)$.
Then, the set $\Pi$ given in (\ref{EQ:PI2}) can be written as
$\Pi=\{(x,y,z): (x,y)\in[0,1]^2, z=T_{\lambda}(x,y), \lambda \in [0,+\infty]\}.$
\end{remark}
\section{Conjunction of $(A|H)$ and  $(A|K)$}
In this section we examine the conjunction of two conditional events in the particular case when $A=B$, that is $(A|H)\wedge (A|K)$.
By setting $P(A|H)=x$, $P(A|K)=y$ and $\prev[(A|H)\wedge (A|K)]=z$, it holds that 
\[
(A|H)\wedge (A|K)=AHK+x\widebar{H}AK+y\widebar{K}AH+z\widebar{H}\,\widebar{K}\in\{1,0,x,y,z\}.
\]
\begin{theorem}\label{THM:A=B}
	Let $A, H, K$ be three logically independent  events, with $H\neq \emptyset$, $K\neq \emptyset$. The set $\Pi$ of all coherent assessments $(x,y,z)$ on the family  $\F=\{A|H,A|K,(A|H)\wedge (A|K)\}$ is given by 
	\begin{equation}\label{EQ:PIA=B}
	\Pi=\{(x,y,z): (x,y)\in[0,1]^2,  T_P(x,y)= xy\leq  z\leq \min\{x,y\}=T_{M}(x,y)\}.
	\end{equation}
\end{theorem}
\begin{proof}
	Let $\M=(x,y,z)$ be a prevision assessment on $\F$.
	The constituents associated with the pair $(\F,\M)$ and contained in $H \vee K$ are:
	$ C_1=AHK$,     $C_2=\widebar{A}HK$, $C_3=\widebar{A}\widebar{H}K$, $C_4=\widebar{A}H\widebar{K}$,
	$C_5=A\widebar{H}K$,  $C_6=AH\widebar{K}$.
	The associated points $Q_h$'s are $Q_1=(1,1,1), Q_2=(0,0,0), Q_3=(x,0,0), Q_4=(0,y,0), Q_5=(x,1,x), Q_6=(1,y,y)$. With the further constituent $C_0=\widebar{H}\widebar{K}$ it is associated the point $Q_0=\mathcal{M}=(x,y,z)$. 		
	Considering the convex hull $\I$ (see Figure \ref{FIG:IEA1}) of $Q_1, \ldots, Q_6$, a necessary condition for the coherence of the prevision assessment $\M=(x,y,z)$ on $\F$ is that $\M \in \I$, that is the following system  must be solvable
	\[
	(\Sigma) \left\{
	\begin{array}{l}
	\lambda_1+x\lambda_3+x\lambda_5+\lambda_6=x,\;\;
	\lambda_1+y\lambda_4+\lambda_5+y\lambda_6=y,\;\;
	\lambda_1+x\lambda_5+y\lambda_6=z,\\
	\sum_{h=1}^6\lambda_h=1,\;\;
	\lambda_h\geq 0,\; h=1,\ldots,6.
	\end{array}
	\right.
	\]
	First of all, we observe that solvability of $(\Sigma)$ requires that $z\leq x$ and $z\leq y$, that is $z\leq \min\{x,y\}$. We now verify that $(x,y,z)$, with $(x,y)\in[0,1]^2$ and $z=\min\{x,y\}$, is  coherent. We distinguish two cases: $(i)$ $x\leq y$ and $(ii)$ $x> y$. \\		
	Case $(i)$. 
	In this case $z=\min\{x,y\}=x$. If $y=0$ 
	the system $(\Sigma)$ becomes
	\[ 
	\begin{array}{l}
	\lambda_1+\lambda_6=0,\;\;
	\lambda_1+\lambda_5=0,\;\;
	\lambda_1=0,\;
	\lambda_2+\lambda_3+\lambda_4=1,\;\;
	\lambda_h\geq 0,\;\; h=1,\ldots,6.
	\end{array}
	\]
	which is clearly solvable. In particular there exist solutions with $\lambda_2>0,\lambda_3>0, \lambda_4>0$, by 
	Theorem \ref{CNES-PREV-I_0-INT},
	as the set $I_0$ is empty the solvability of $(\Sigma)$ is sufficient for coherence of the assessment $(0,0,0)$. 
	If $y>0$ the system $(\Sigma)$ is solvable and a solution is $	\Lambda=(\lambda_1,\ldots, \lambda_6)=(x,\frac{x(1-y)}{y},0,\frac{y-x}{y},0,0)$.
	We observe that, if $x>0$, then $\lambda_1>0$ and $I_0=\emptyset$ because $\C_1=HK\subseteq H\vee K$, so that $\M=(x,y,x)$ is coherent. If $x=0$ (and hence $z=0$), then $\lambda_4=1$ and $I_0\subseteq \{2\}$. Then, as the sub-assessment $P(A|K)=y$ is coherent, it follows that the assessment $\M=(0,y,0)$ is coherent too.\\
	Case $(ii)$. The system is solvable and a solution is $
	\Lambda=(\lambda_1,\ldots, \lambda_6)=(y,\frac{y(1-x)}{x},\frac{x-y}{x},0,0,0).$
	We observe that, if $y>0$, then $\lambda_1>0$ and $I_0=\emptyset$ because $\C_1=HK\subseteq H\vee K$, so that $\M=(x,y,y)$ is coherent. If $y=0$ (and hence $z=0$), then $\lambda_3=1$ and $I_0\subseteq \{1\}$. Then, as the sub-assessment $P(A|H)=x$ is coherent, it follows that the assessment $\M=(x,0,0)$ is coherent too.
	Thus, for every $(x,y)\in[0,1]^2$, the assessment $(x,y,\min\{x,y\})$ is coherent  and, as $z\leq \min\{x,y\}$,   the upper bound on $z$ is $\min\{x,y\}=T_M(x,y)$. \\
	We now verify that $(x,y,xy)$, with $(x,y)\in[0,1]^2$ is coherent; moreover we will show that  $(x,y,z)$, with $z<xy$, is not coherent, in other words the lower bound for $z$ is $xy$. 
	First of all, we observe that $\M=(1-x)Q_4+xQ_6$, so that a solution of $(\Sigma)$ is $\Lambda_1=(0,0,0,1-x,0,x)$.
	Moreover, $\M=(1-y)Q_3+yQ_5$, so that another solution is $\Lambda_2=(0,0,1-y,0,y,0)$. Then 
$
	\Lambda=\frac{\Lambda_1+\Lambda_2}{2}=(0,0,\frac{1-y}{2},\frac{1-x}{2},\frac{y}{2},\frac{x}{2})
$
	is a solution of $(\Sigma)$ such that $I_0=\emptyset$. Thus the assessment $(x,y,xy)$ is coherent for every $(x,y)\in[0,1]^2$. 
	In order to verify that  $xy$ is the lower bound on $z$ we observe that the points $Q_3,Q_4,Q_5,Q_6$ belong to a plane $\pi$ of equation: 
	$yX+xY-Z=xy$,	where $X,Y,Z$ are the axis' coordinates. Now, by considering the function $f(X,Y,Z)= yX+xY-Z$, we observe that for each constant $k$ the equation $f(X,Y,Z)=k$ represents a plane which is parallel to $\pi$ and coincides with $\pi$ when  $k=xy$. We also observe that 
	$f(Q_1)=f(1,1,1)=x+y-1=T_L(x,y)\leq xy=T_P(x,y)$,  
	$f(Q_2)=f(0,0,0)=0 \leq xy=T_P(x,y)$,  and
	$f(Q_3)=f(Q_4)=f(Q_5)=f(Q_6)= xy=T_P(x,y)$.
	Then, for every $\P=\sum_{h=1}^6\lambda_hQ_h$, with $\lambda_h\geq 0$ and $\sum_{h=1}^6\lambda_h=1$, that is $\P\in \I$, it holds that 
$
	f(\P)=f\big(\sum_{h=1}^6\lambda_hQ_h\big)=\sum_{h=1}^6\lambda_hf(Q_h)\leq  xy.
$
	On the other hand, given any $a>0$, by considering  $\P=(x,y,xy-a)$ it holds that 
$
	f(\P)=f(x,y,xy-a)=xy+xy-xy+a= xy+a>xy.
	$
	Therefore, for any given $a>0$ the assessment $(x,y,xy-a)$ is not coherent because $(x,y,xy-a)\notin \I$. Then, the lower bound on $z$ is $xy=T_{P}(x,y)$. Finally, the set of all coherent assessments $(x,y,z)$ on $\F$ is the set $\Pi$ in 
	(\ref{EQ:PIA=B}).
	\qed
\end{proof}
\begin{figure}[tpbh]
	\centering
	\vspace{-1cm}
	\includegraphics[width=0.70\linewidth]{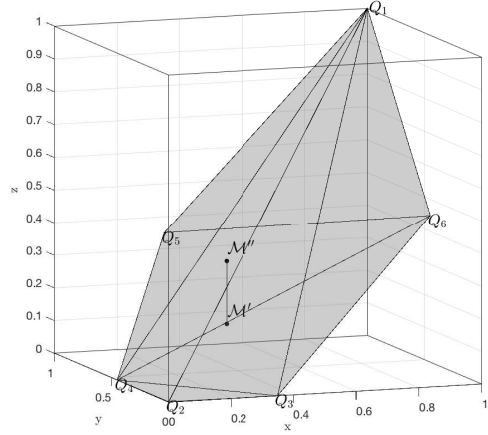}
	\vspace{-0.5cm}
	\caption{Convex hull $\I$ of  the points $Q_1, Q_2,Q_3, Q_4, Q_5,Q_6$.  
		$\M'=(x,y,z'), \M''=(x,y,z'')$, where $(x,y)\in[0,1]^2$, $z'=xy$, $z''=\min\{x,y\}$. In the figure the numerical  values are: $x=0.35$, $y=0.45$, $z'=0.1575$, and  $z''=0.35$.}
	\label{FIG:IEA1}
	\vspace{-0.5cm}
\end{figure}		
Based on Theorem \ref{THM:A=B}, we can give an analogous version for the Theorem \ref{THM_TNORM2} (when $A=B$).
\begin{theorem}\label{THM_TNORM2A=B}
	Let us consider the conjunction $(A|H)\wedge(A|K)$, with $A, H, K$ logically independent and with $P(A|H)=x$, $P(A|K)=y$.
	Moreover, given any $\lambda  \in[0,1]$, let $T_{\lambda}$ be the  Frank t-norm with parameter $\lambda$. Then,  the assessment
	$z=T_{\lambda}(x,y)$ on $(A|H)\wedge(A|K)$ is a coherent extension of $(x,y)$ on $\{A|H,A|K\}$; moreover   $(A|H)\wedge(A|K)=T_{\lambda}(A|H, A|K)$. Conversely, given any coherent extension  $z=\prev[(A|H)\wedge(A|K)]$ of $(x,y)$, there exists $\lambda \in[0,1]$ such that $z=T_{\lambda}(x,y)$.
\end{theorem}
The next result follows from Theorem \ref{THM:A=B} when $H$, $K$ are incompatible.
\begin{theorem}\label{THM:SETPROD}
	Let $A, H, K$ be three events, with  $A$  logically independent from both $H$ and $K$, with $H\neq \emptyset$, $K\neq \emptyset$, $HK=\emptyset$. The set $\Pi$ of all coherent assessments $(x,y,z)$ on the family  $\F=\{A|H,A|K,(A|H)\wedge (A|K)\}$ is given by $\Pi=\{(x,y,z): (x,y)\in[0,1]^2,  z= xy=T_P(x,y)\}.$
\end{theorem}
\begin{proof}
	We observe that 
	\[
	\begin{small}
	(A|H)\wedge (A|K)=\left\{\begin{array}{ll}
	0, &\mbox{ if }  \widebar{A}\widebar{H}K \vee \widebar{A}H\widebar{K}   \mbox{ is true,}\\
	x, &\mbox{ if }  \widebar{H}AK \mbox{ is true,}\\
	y, &\mbox{ if }  AH\widebar{K}\mbox{ is true,}\\
	z, &\mbox{ if }  \widebar{H}\widebar{K} \mbox{ is true.}\\
	\end{array}
	\right.
	\end{small}
	\]
Moreover, as $HK=\emptyset$,  
the points $Q_h$'s are $(x,0,0), (0,y,0), (x,1,x), (1,y,y)$, which coincide with the points $Q_3,\ldots, Q_6$ of the case $HK\neq \emptyset$. Then,  as shown in the proof of Theorem \ref{THM:A=B}, the condition  $\M=(x,y,z)$ belongs to the convex hull  of  $(x,0,0), (0,y,0), (x,1,x), (1,y,y)$ amounts to the condition $z=xy$.
\qed
\end{proof}	
\begin{remark}
From Theorem \ref{THM:SETPROD}, when $HK=\emptyset$ it holds that 
$
(A|H)\wedge (A|K)=(A|H)\cdot (A|K)=T_P(A|H,A|K),
$
where $x=P(A|H)$ and $y=P(A|K)$.
\end{remark}	

\section{Further Results on Frank t-norms}
In this section we give some results which concern Frank t-norms and the family 
$\F=\{\C_{1},\C_{2},\C_{3}, \C_{12}, \C_{13}, \C_{23}, \C_{123}\}$. We recall that, given any t-norm $T(x_1,x_2)$ it holds that $T(x_1,x_2,x_3)=T(T(x_1,x_2),x_3)$.
\subsection{On the Product t-norm}
\begin{theorem}\label{THM:PROD}
Assume that  the events $E_1, E_2, E_3, H_1, H_2, H_3$  are logically independent, with $H_1\neq \emptyset, H_2\neq \emptyset, H_3\neq \emptyset$.
If  the  assessment  $\mathcal{M}=(x_1,x_2,x_3,x_{12},x_{13},x_{23},x_{123})$ on
$\F=\{\C_{1},\C_{2},\C_{3}, \C_{12}, \C_{13}, \C_{23}, \C_{123}\}$ is such that $(x_1,x_2,x_3)\in[0,1]^3$, $x_{ij}=T_{1}(x_i,x_j)=x_ix_j$, $i\neq j$,  and $x_{123}=T_{1}(x_1,x_2,x_3)=x_1x_2x_3$, then
$\M$ is coherent. Moreover, $\mathcal{C}_{ij}=T_{1}(\mathcal{C}_i,\mathcal{C}_j)=\mathcal{C}_i\mathcal{C}_j$, $i\neq j$, and  $\mathcal{C}_{123}=T_{1}(\mathcal{C}_1,\mathcal{C}_2,\mathcal{C}_3)=\mathcal{C}_{1}\mathcal{C}_{2}\mathcal{C}_{3}$. 
\end{theorem}
\begin{proof}
From Remark \ref{REM:INEQPI}, the coherence of $\M$ amounts to the inequalities in (\ref{EQ:INEQPI}).
As $x_{ij}=T_{1}(x_i,x_j)=x_ix_j$, $i\neq j$,  and $x_{123}=T_{1}(x_1,x_2,x_3)=x_1x_2x_3$, 
the inequalities (\ref{EQ:INEQPI}) become
\begin{equation}
\begin{array}{ll}
\max\{0,x_1(x_2+x_3-1),x_{2}(x_1+x_3-1),x_3(x_1+x_2-1)\}\,\;\leq  \;x_{1}x_2x_3\;\leq \\
\leq\;\; 
\min\{x_{1}x_2,x_{1}x_3,x_{2}x_3,(1-x_1)(1-x_2)(1-x_3)+x_1x_2x_3\}.
\end{array}
\end{equation}
Thus, by recalling that $x_i+x_j-1\leq x_ix_j$, the inequalities are satisfied and hence $\M$ is coherent. Moreover, from (\ref{EQ:CONJUNCTION}) and (\ref{EQ:CONJUNCTION3}) it follows that 
$
\C_{ij}=T_{1}(\C_i,\C_j)=\C_i\C_j$,  $i\neq j$, and  $\C_{123}=T_{1}(\C_1,\C_2,\C_3)=\C_{1}\C_{2}\C_{3}$.\qed 
\end{proof}
\subsection{On the Minimum t-norm}
\begin{theorem}\label{THM:MIN}
	Assume that  the events $E_1, E_2, E_3, H_1, H_2, H_3$  are logically independent, with $H_1\neq \emptyset, H_2\neq \emptyset, H_3\neq \emptyset$.
	If  the  assessment  $\mathcal{M}=(x_1,x_2,x_3,x_{12},x_{13},x_{23},x_{123})$ on
	$\F=\{\C_{1},\C_{2},\C_{3}, \C_{12}, \C_{13}, \C_{23}, \C_{123}\}$ is such that $(x_1,x_2,x_3)\in[0,1]^3$, $x_{ij}=T_{M}(x_i,x_j)=\min\{x_i,x_j\}$, $i\neq j$,  and $x_{123}=T_{M}(x_1,x_2,x_3)=\min\{x_1,x_2,x_3\}$, then
	$\M$ is coherent. Moreover, $\C_{ij}=T_{M}(\mathcal{C}_i,\mathcal{C}_j)=\min\{\mathcal{C}_i,\mathcal{C}_j\}$, $i\neq j$, and  $\C_{123}=T_{M}(\mathcal{C}_1,\mathcal{C}_2,\mathcal{C}_3)=\min\{\mathcal{C}_{1},\mathcal{C}_{2},\mathcal{C}_{3}\}$. 
\end{theorem}
\begin{proof}
	From Remark \ref{REM:INEQPI}, the coherence of $\M$ amounts to the inequalities in (\ref{EQ:INEQPI}).
Without loss of generality, we assume that $x_1\leq x_2\leq x_3$.  Then 
 $x_{12}=T_{M}(x_1,x_2)=x_1$,
  $x_{13}=T_{M}(x_1,x_3)=x_1$,
   $x_{23}=T_{M}(x_2,x_3)=x_2$, and $x_{123}=T_{M}(x_1,x_2,x_3)=x_1$. 
	The inequalities (\ref{EQ:INEQPI}) become
	\begin{equation}\label{EQ:MIN}
\begin{array}{ll}
\max\{0,x_{1},x_{1}+x_2-x_3\}=x_1\,\;\leq  \;x_{1}\;\leq x_1=
\min\{x_{1},x_{2},1-x_3+x_{1}\}.
\end{array}	\end{equation}
	Thus,  the inequalities are satisfied and hence $\M$ is coherent. Moreover, from (\ref{EQ:CONJUNCTION}) and (\ref{EQ:CONJUNCTION3}) it follows that 
$\C_{ij}=T_{M}(\C_i,\C_j)=\min\{\C_i,\C_j\}$,  $i\neq j$, and $\C_{123}=T_{M}(\C_1,\C_2,\C_3)=\min\{\C_{1},\C_{2},\C_{3}\}$.
\qed \end{proof}
\begin{remark}
As we can see from 
 $(\ref{EQ:MIN})$ and Corollary \ref{COR:PIFOR3}, the assessment $x_{123}=\min\{x_1,x_2,x_3\}$ is the unique coherent extension on $\C_{123}$ of the assessment $
 (x_1,x_2,x_3,\min\{x_1,x_2\},\min\{x_1,x_3\},\min\{x_2,x_3\})$ on
 $\{\C_{1},\C_{2},\C_{3}, \C_{12}, \C_{13}, \C_{23}\}$. \\ We also notice that, if $\C_1\leq \C_2\leq \C_3$, then $\C_{12}=\C_1$, $\C_{13}=\C_1$, $\C_{23}=\C_2$, and $\C_{123}=\C_1$. Moreover,   $x_{12}=x_1$,
	$x_{13}=x_1$,
	$x_{23}=x_2$, and $x_{123}=x_1$. 
\end{remark}	

\subsection{On Lukasiewicz t-norm}
We observe that in general the results of Theorems \ref{THM:PROD} and \ref{THM:MIN} do not hold for the Lukasiewicz t-norm  (and hence for any given Frank t-norm), as shown in the example below. We recall that $T_L(x_1,x_2,x_3)=\max\{x_1+x_2+x_3-2,0\}$.
\begin{example}
The assessment
$(x_1,x_2,x_3,T_L(x_1,x_2),T_L(x_1,x_3),T_L(x_2,x_3)$, $T_L(x_1,x_2,x_3))$
on the family $\F=\{\C_{1},\C_{2},\C_{3}, \C_{12}, \C_{13}, \C_{23}, \C_{123}\}$, with 	$(x_1,x_2,x_3)=(0.5,0.6,0.7)$ is not coherent.
Indeed, by observing that  $T_L(x_1,x_2)=0.1$ 
$T_L(x_1,x_3)=0.2$, $T_L(x_2,x_3)=0.3$, and $T_L(x_1,x_2,x_3)=0$, formula (\ref{EQ:INEQPI}) becomes
$
\max\{0, 0.1+0.2-0.5,0.1+0.3-0.6,0.2+0.3-0.7\}\,\;\leq  \;0\;
\leq\;\; 
\min\{0.1,0.2,0.3,1-0.5-0.6-0.7+0.1+0.2+0.3\}$,
that is: 
$
\max\{0, -0.2\}\,\;\leq  \;0\;\leq 
\min\{0.1,0.2,0.3,-0.2\};
$
thus the inequalities are not satisfied and the assessment is not coherent.
\end{example}
More in general we have
\begin{theorem}\label{THM:LUK}
The assessment
$(x_1,x_2,x_3,T_L(x_1,x_2),T_L(x_1,x_3),T_L(x_2,x_3))$
on the family  $\F=\{\C_{1},\C_{2},\C_{3}, \C_{12}, \C_{13}, \C_{23}\}$, with 
$T_L(x_1,x_2)>0$, ${T_L(x_1,x_3)>0}$, ${T_L(x_2,x_3)>0}$ is coherent if and only if 
 $x_1+x_2+x_3-2\geq 0$. Moreover, when $x_1+x_2+x_3-2\geq 0$ the  unique coherent extension $x_{123}$ on $\C_{123}$ is $x_{123}=T_L(x_1,x_2,x_3)$.
\end{theorem}
\begin{proof}
We distinguish two cases: $(i)$ 	$x_1+x_2+x_3-2< 0$; $(ii)$ 	$x_1+x_2+x_3-2\geq 0$.\\
Case $(i)$.
From	
(\ref{EQ:SYSTEMPISTATEMENT}) the inequality 
$1-x_1-x_2-x_3+x_{12}+x_{13}+x_{23}\geq 0$
is not satisfied because 
$
1-x_1-x_2-x_3+x_{12}+x_{13}+x_{23}=x_{1}+x_2+x_3-2<0.
$
Therefore the assessment is not coherent.\\
Case $(ii)$.
We set $x_{123}=T_L(x_1,x_2,x_3)=x_1+x_2+x_3-2$.
Then, by observing that   $0<x_i+x_j-1\leq x_1+x_2+x_3-2$, $i\neq j$,   formula (\ref{EQ:INEQPI}) becomes 
$\max\{0,x_{1}+x_2+x_3-2\}\,\;\leq  \;x_{1}+x_2+x_3-2\;
\leq\;\; 
\min\{x_{1}+x_2-1,x_{1}+x_3-1,x_{2}+x_3-1,x_1+x_2+x_3-2\}$,
that is:
$
\;x_{1}+x_2+x_3-2\;\leq \;x_{1}+x_2+x_3-2\;\leq \;x_{1}+x_2+x_3-2.
$
Thus,  the inequalities are satisfied and the
 assessment 
$
 (x_1,x_2,x_3,T_L(x_1,x_2),T_L(x_1,x_3),T_L(x_2,x_3), T_L(x_1,x_2,x_3))$ on $\{\C_{1},\C_{2},\C_{3}, \C_{12}, \C_{13}, \C_{23},\C_{123}\}$
  is coherent and the sub-assessment 
 $
  (x_1,x_2,x_3,T_L(x_1,x_2),T_L(x_1,x_3),T_L(x_2,x_3))
  $ on 
  $\F$ is coherent too.
	\qed \end{proof}
A result related with Theorem \ref{THM:LUK} is given below.
\begin{theorem}
	If the assessment
$(x_1,x_2,x_3,T_L(x_1,x_2),T_L(x_1,x_3),T_L(x_2,x_3)$, $T_L(x_1,x_2,x_3))$
on the family $\F=\{\C_{1},\C_{2},\C_{3}, \C_{12}, \C_{13}, \C_{23}, \C_{123} \}$, is such that $T_L(x_1,x_2,x_3)>0$, then the assessment is  coherent.	
\end{theorem}
\begin{proof}
	We observe that $T_L(x_1,x_2,x_3)=x_1+x_2+x_3-2>0$; then 
	$x_i>0$, $i=1,2,3$, and  $0<x_i+x_j-1\leq x_1+x_2+x_3-2$, $i\neq j$.   
	Then formula (\ref{EQ:INEQPI})) becomes: 
	$\;\;\max\{0,x_{1}+x_2+x_3-2\}\,\;\leq  \;x_{1}+x_2+x_3-2\;\leq\\
	\leq\;\; 
	\min\{x_{1}+x_2-1,x_{1}+x_3-1,x_{2}+x_3-1,x_1+x_2+x_3-2\},$
	that is: \\ $x_{1}+x_2+x_3-2\;\leq \;x_{1}+x_2+x_3-2\;\leq \;x_{1}+x_2+x_3-2.$
\\	Thus,  the inequalities are satisfied and the assessment is coherent.
\qed \end{proof}	
\section{Conclusions} 
We have studied the relationship between the notions of conjunction and of Frank t-norms. 
We have shown that, under logical independence of events and coherence of prevision assessments, for a suitable $\lambda \in [0,+\infty]$ it holds that $\prev((A|H) \wedge (B|K))= T_\lambda(x,y)$ and $(A|H) \wedge (B|K)= T_\lambda(A|H,B|K)$.  Then, we have considered the case $A=B$, by determining the set of all coherent assessment $(x,y,z)$ on $(A|H,B|K,(A|H) \wedge (A|K))$. 
We have shown that, under coherence, for a suitable $\lambda \in [0,1]$ it holds that 
$(A|H) \wedge (A|K)= T_\lambda(A|H,A|K)$.
We have also studied the particular case where $A=B$ and $HK=\emptyset$. Then, we have considered the conjunction of three conditional events and we have shown that the prevision assessments produced by the Product t-norm, or the Minimum t-norm, are coherent. Finally, we have examined the 
Lukasiewicz t-norm and we have shown, by a counterexample, that coherence in general is not assured. 
We have given some  conditions for coherence when the prevision assessments are based on the Lukasiewicz t-norm. Future work should concern the deepening and generalization of the results of this paper.  
\\ \ \\
{\bf Acknowledgments}. We thank three anonymous referees for their useful comments. 


\end{document}